\documentclass[12pt]{article}
\usepackage{amsmath,amssymb,amsfonts,amsthm}
\usepackage{xspace,xargs,enumitem,graphics,xcolor,graphicx}
\usepackage[sort&compress,square,numbers]{natbib}
\usepackage{hyperref}
\usepackage{etoolbox}
\usepackage{algorithm}
\usepackage{algpseudocode}
\makeatletter
\preto{\@verbatim}{\topsep=0pt \partopsep=0pt }
\makeatother

\usepackage{cleveref}

\def\Cal#1{\mathcal{#1}}

\def\Grp#1{\left(#1\right)}
\def\Cbr#1{\left\{#1\right\}}
\def\Sbr#1{\left[#1\right]}

\def\gv{\,|\,}
\def\cf#1{\mathbf{1}\!\Cbr{#1}}

\def\toi{\to\infty}
\def\Linf{\varliminf}
\def\Lsup{\varlimsup}
\def\toi{\to\infty}
\def\Iff{\Longleftrightarrow}   

\def\Coms{\mathbb{C}}
\def\Ints{\mathbb{Z}}
\def\Reals{\mathbb{R}}

\newcommandx*\inum[2][1=1]{#2_{#1}, #2_{\number\numexpr#1+1}, \ldots}
\newcommandx*\cset[3][1=1]{#2_{#1} \ldots #2_{#3}}
\newcommandx*\cum[3][1=1]{#2_{#1}+\cdots+#2_{#3}}
\newcommandx*\eno[3][1=1]{#2_{#1},\ldots,#2_{#3}}
\newcommandx*\seqop[4][1=1]{#2_{#1}#3\cdots#3 #2_{#4}}

\def\sumoi#1{\sum_{#1=1}^\infty}
\def\sumzi#1{\sum_{#1=0}^\infty}

\newtheorem{theorem}{Theorem}[section]
\crefname{theorem}{theorem}{theorems}
\newtheorem{prop}[theorem]{Proposition}
\crefname{prop}{proposition}{propositions}
\newtheorem{lemma}[theorem]{Lemma}
\crefname{lemma}{lemma}{lemmas}

\theoremstyle{definition}

\crefname{example}{example}{examples}
\theoremstyle{definition}
\newtheorem{definition}{Definition}
\theoremstyle{remark}
\newtheorem*{remark}{Remark}

\setlist{itemsep=0ex, topsep=1ex, parsep=0ex}
\newlist{subenum}{enumerate}{3}
\setlist[subenum,1]{label={\alph*}),leftmargin=5ex}

\def\Sp#1{\sp{(#1)}}

\def\nth#1{\frac{1}{#1}}
\def\lfrac#1#2{#1/#2}     

\def\mean{\operatorname{E}}
\def\prob{\operatorname{P}}
\def\rx{\epsilon}
\def\udisc{\mathbb{D}}

\makeatletter
\def\@cleandot{\@ifnextchar.{}{\@ifnextchar,{.}{\@ifnextchar;{.}{\@ifnextchar?{.}{\@ifnextchar:{.}{\@ifnextchar!{.}{\@ifnextchar'{.}{\@ifnextchar){.}{\@ifnextchar({.}{\@ifnextchar-{.}{\@ifnextchar\/{.}{\@ifnextchar\\{.}{.\ }}}}}}}}}}}}}
\def\pmf{p.m.f\@cleandot}
\def\rhs{r.h.s\@cleandot}
\makeatother

\title{}  

\author{
}

\date{\today}

\begin{document}
\begin{center}
  \large
  \textbf{On difference in word frequencies in a symmetric Bernoulli
    process}
  \\[1.5ex]
  \normalsize
  Zhiyi Chi \\
  Department of Statistics\\
  University of Connecticut,
  Storrs, CT 06269, USA, \\[.5ex]
  E-mail: zhiyi.chi@uconn.edu\\[1ex]
  \today
\end{center}

\begin{abstract}
  In a symmetric Bernoulli process, all binary strings, or ``words''
  of the same length have the same long term frequency.  However,
  between two such words, one may have a ``frequency advantage'' in 
  the sense that in any long enough segment of the Bernoulli process,
  the probability that the word occurs more times than the other word
  is greater than the probability the other way around.  To
  characterize the frequency advantage in the long run, the 
  asymptotics of the difference between the two probabilities as the
  length of the segment of the Bernoulli process tends to infinity is
  derived.
  
  \medbreak\noindent

  \medbreak\noindent
\end{abstract}
\section{Introduction}
Let $(X_i,i\ge1)$ be a Bernoulli process with $\prob\{X_i = 0\} =
1/2$.  For a word $a$, i.e., a finite sequence of 0's and 1's, denote
by $C_n(a)$ the number of times it occurs in  $\eno X n$.  Using
combinatorial method, Levin \cite {levin:24:arxiv} showed that for
$n>2$, $\prob\{C_n(10)> C_n(11)\} > \prob\{C_n(10) < C_n(11)\}$, which
means the event that 10 appears more times than 11 in $\eno X n$ has
a better chance than the event the other way around.  We will
say that 10 has a ``frequency advantage'' over 11.  The result
is somewhat surprising as the two words have the same long-run
frequency in $\inum X$.  Naturally, a question is whether similar
comparisons can be made to other pairs of words.  In general, for two
different words $a$ and $b$, a direct comparison of
$\prob\{C_n(a)>C_n(b)\}$ and $\prob\{C_n(a) < C_n(b)\}$ seems
intractable.  However, if the asymptotics of
\[
  \varphi_n(a,b) :=
  \prob\{C_n(a)>C_n(b)\} - \prob\{C_n(a)<C_n(b)\} \quad \text{as~}
  n\toi
\]
can be obtained, then we know which word has a frequency advantage
in the long run.

Let us first fix some notation.  For $x=(\eno x n)$, where $0\le n\le
\infty$, denote $|x|=n$ and refer to it as the length of $x$.  If
$n=0$, then $x$ is an empty word denoted by $\rx$.  For $k,m\ge1$, the
``look-back'' segment of length $m$ at location $k$ in $x$ is defined
to be
\[
  (x)^m_k =
  \begin{cases}
    (\eno[k-m+1] x k) & \text{if~} m\le k\le n
    \\
    \rx & \text{else}.
  \end{cases}
\]
If $y$ is a word of length $m\ge1$, denote 
\[
  N_x(y) = \sum^n_{k=m} \cf{(x)^m_k=y}.
\]
Then for each integer $n\ge0$, $C_n(y) = N_{(\eno X n)}(y)$.

Two cases can be quickly resolved for $\varphi_n(a,b)$.  First, from
$(C_n(a)/n, C_n(b)/n)\to (2^{-|a|}, 2^{-|b|})$ almost surely, if
$|a|\ne |b|$, then $\varphi_n(a,b)\to1$ or $-1$, depending on whether
or not $a$ is the shorter one.  Next, for $x=(\eno x n)\in \{0,1\}^n$,
denote by  $x^R=(x_n, \ldots,
x_1)$ its time reversal and by $\bar x = (\eno{1-x} n)$ its
``conjugate''.  If $\tilde{\ }$ is any one of the mappings $x\mapsto
x^R$,  $x\mapsto\bar x$, and $x\mapsto
\bar x^R$, then by $N_x(\cdot) = N_{\tilde x}(\tilde \cdot)$ and
$X:=(\eno X n)\sim \tilde X$,  $(C_n(a), C_n(b)) = (N_{\tilde
  X}(\tilde a), N_{\tilde X}(\tilde b)) \sim (N_X(\tilde a),
N_X(\tilde b))$, giving $\prob\{C_n(a)>C_n(b)\} = \prob\{C_n(\tilde a)
> C_n(\tilde b)\}$.  In particular, letting $b=\tilde a$, by $\tilde
b=a$, $\varphi_n(a, \tilde a)\equiv0$.  Thus, to see between $a$ and
$b$, which one has a frequency advantage over the other, we only need
to consider the case where
\begin{align} \label{e:assumption}
  b \not\in \{a, a^R, \bar a, \bar a^R\} \quad\text{and}
  \quad |a| =|b|.
\end{align}

Let $a,b\in\{0,1\}^L$ satisfy \eqref{e:assumption}.  Then $L>1$.  Let
$T_k$ be the times when $a$ or $b$ occurs in $\inum X$.
Specifically, letting $T_0=0$,
\begin{align} \label{e:Tk}
  T_k=
  \min\{t>T_{k-1}: (X)^L_t = a \text{~or~} b\}, \quad k\ge1.
\end{align}
Since $((X)^L_{L+t-1}, t\ge1)$ is a time-homogeneous irreducible
Markov chain in $\{0,1\}^L$, then $\prob\{T_k<\infty, k\ge1\}=1$.
Define
\begin{align} \label{e:Y-X}
  Y_k =
  \begin{cases}
    1 & \text{if~} (X)^L_{T_k}=a,
    \\
    -1 & \text{else,}
  \end{cases}
\end{align}
which form a time-homogeneous irreducible Markov chain.  Since
$C_n(u)/n\to 1/2^L$ a.s.\ for any $u\in \{0,1\}^L$, the stationary
distribution of $Y_n$ is uniform on $\pm1$.  As a result,
$\prob\{Y_2=1\gv Y_1=1\} = \prob\{Y_2=-1\gv Y_1=-1\}$.  Put
\begin{align} \label{e:tran-p}
  p = \prob\{Y_2=1\gv Y_1=1\}=1-q
\end{align}
and
\begin{align} \label{e:mean-gap}
  \mu_{c,e}=\mean(T_2 - T_1, Y_2=e\gv Y_1=c).
\end{align}
A general result is as follows.
\begin{theorem} \label{t:darboux-1}
  Fix $a,b\in\{0,1\}^L$ satisfying \eqref{e:assumption}.  Let $p_c =
  \prob\{Y_1=c\}$, $c=\pm1$. 
  Then
  \[
    \varphi_n(a,b)
    =
    \frac{(p_1-p_{-1}) 2^L + \mu_{1,-1} - \mu_{-1,1}}
      {\sqrt{2^{L+2} pq\pi n}} + o(1/\sqrt n), \quad n\toi.
  \]
\end{theorem}

\begin{remark}{\ }
  \begin{enumerate}
  \item By $\prob\{Y_1=1\} \ge \prob\{(X)^L_L=a\}$, $p_1>0$.
    Likewise, $p_{-1}>0$.  On the other hand, it will be shown that
    $p\in (0,1)$, which implies $\mu_{c,e}>0$ and hence
    $\lambda_{\pm1}$ and $\mu$ are well defined and positive.
  \item There is a standard method to calculate $p_c$, $p$, and
    $\mu_{c,e}$, $c,e=\pm1$; see \cref{s:first-step}.
  \end{enumerate}
\end{remark}

As an example, let $a=1^s0$ and $b=1^{s+1}$, where for $c\in\{0,1\}$,
$c^s$ denotes the sequence of $s$ consecutive $c$'s.  It is not hard
to check that $p_1 = p_{-1}= p=q=1/2$, $\mu_{-1,1} = \mu_{-1,-1} =
1/2$, and $\mu_{1,1} = \mu_{1,-1}$.  Then by 
\begin{align} \label{e:sum-mean}
  \sum_{c,e=\pm1} \mu_{c,e} = 2^{s+1},
\end{align}
$\mu_{1,-1} = 2^s-1/2$.  Then by \Cref{t:darboux-1},
\[
  \varphi_n(1^s0, 1^{s+1}) \sim \frac{2^s-1}{\sqrt{2^{s+1}\pi n}},
  \quad n\toi.
\]
In particular, $\varphi_n(10,11) \sim 1/\sqrt{4\pi
  n}$.  A proof of \eqref{e:sum-mean} is given in \cref{s:first-step}.

The table below displays comparisons between words of length 3.  For
each pair $(a,b)$, the corresponding $\lambda$ is the coefficient in
the asymptotics $\varphi_n(a,b) = \lambda/\sqrt{\pi n} +
o(1/\sqrt n)$.  Pairs not in the table either fail to satisfy
\eqref{e:assumption}, resulting in $\varphi_n(a,b)\equiv0$ or can be
transformed via time reversal, conjugation, or their composite to one
of the pairs in the table:
\[
  \begin{array}{rcccccc}
    (a,b): & (111,110) & (111,101) & (111,100) &(111,010) & (110,101) &
    (110,010) \\ 
    \lambda: & -\frac3{2\sqrt2} & -\nth{\sqrt 3} & -\frac{\sqrt 3}2
    & -\nth2 & \nth2 & \nth{2\sqrt2}
  \end{array}
\]

It is seen that for all $a,b\in\{0,1\}^3$ satisfying
\eqref{e:assumption}, $\varphi_n(a,b)$ has the same order as $1/\sqrt
n$.  However, this is not always the case for longer words.  For
example, for $a=0001$ and $b=1100$, $\lambda=0$.  From the derivation
in the paper, $\varphi_n(0001, 1100)$ is either of the same order as
$1/n^{k+1/2}$ for some positive integer $k$, or $O(\varrho^{-n})$ for
some $\varrho>1$.  Unfortunately, we did not have enough computing
capacity to determine which is the case.

The main result of the paper is the following and will be proved in \cref{s:gf-step1,s:gf-step2,s:singular}.

\begin{theorem} \label{t:darboux-0}
  Fix $v,a,b\in\{0,1\}^L$ with $a$ and $b$ satisfying
  \eqref{e:assumption}.  Let
  \[
    \psi^v_n(a,b) = \prob\{C_n(a)>C_n(b)\gv (X)^L_L=v\}-1/2
  \]
  and
  \begin{align}\label{e:Lambda-deriv-diff}
    p^v_c = \prob\{Y_1=c\gv (X)^L_L=v\}, \quad c=\pm1.
  \end{align}
  Let
  \begin{align} \label{e:beta}
    \beta^v_{a,b}=
    (\lfrac p2 - p^v_{-1}) 2^L +\frac{\mu_{1,-1} - \mu_{-1,1}}2.
  \end{align}
  Then
  \begin{align} \label{e:psi-n}
    \psi^v_n(a,b)=\frac{\beta^v_{a,b}}{\sqrt{2^{L+2}pq\pi n}} +
    o(1/\sqrt n),
    \quad n\toi.
  \end{align}
\end{theorem}

Assuming \Cref{t:darboux-0} is correct for now, then
\begin{align*}
  \varphi^v_n(a,b)
  &:=
  \prob\{C_n(a)>C_n(b)\gv (X)^L=v\} -\prob\{C_n(b)>C_n(a)\gv (X)^L=v\}
  \\
  &=
  \psi^v_n(a,b) - \psi^v_n(b,a)
  =
  \frac{\beta^v_{a,b}-\beta^v_{b,a}}{\sqrt{2^{L+2}pq\pi n}} +
  o(1/\sqrt n).
\end{align*}
By symmetry, $\beta^v_{b,a}$ is obtained from \eqref{e:beta} by
switching $1$ and $-1$, giving
\[
  \beta^v_{a,b} - \beta^v_{b,a} =(p^v_1 - p^v_{-1})2^L + \mu_{1,-1}
  - \mu_{-1,1}. 
\]
Since $\varphi_n(a,b)$ and $p_{\pm1}$ are the averages of
$\varphi^v_n(a,b)$ and $p^v_{\pm1}$ over $v\in \{0,1\}$, respectively,
then \Cref{t:darboux-1} follows.

Two corollaries are mentioning.  First, let $v=a$ in
\Cref{t:darboux-1}.  Then by \eqref{e:sum-mean} and
$p^a_1 = 1 = 1-p^a_{-1}$, 
\[
  \varphi^a_n(a,b)
  =\frac{\mu_{1,1}+\mu_{-1,-1} + 2\mu_{1,-1}}{
    \sqrt{2^{L+2} pq\pi n}} + o(1/\sqrt n).
\]
In other words, if $a$ is the initial word of $\inum X$, then it has
frequency advantage over $b$ for $n\gg1$.  Second, since
$\varphi^v_n(a,b) + \varphi^v_n(b,a) = -\prob\{C_n(a) = C_n(b)\gv
(X)^L_L=v\}$ and $p^v_1 + p^v_{-1}=1$, then by \Cref{t:darboux-0},
\[
  \prob\{C_n(a) = C_n(b)\gv (X)^L_L=v\}
  = \sqrt{\frac{2^{L-2} q}{p\pi n}} + o(1/\sqrt n), \quad n\toi.
\]

\section{Preliminaries} \label{s:prelim}
This section establishes some basic facts, one of them being the
transition probabilities $p$ and $q$ in \eqref{e:tran-p} are positive.
Denote by $\{0,1\}^\#$ is set of non-empty words.  For $u=(\eno u n)$
and $v=(\eno v m)$, denote by $uv = (\eno u n, \eno v m)$ their
concatenation.  Denote $u^0=\rx$ and $u^s = u^{s-1}u$ for $s>0$.

\begin{prop} \label{p:nonempty}
  Let $a,b\in \{0,1\}^L$.  For $v\in\{0,1\}^\#$, define 
  \begin{align} \label{e:Pi-path}
    \Pi^v_{a,b}
    =
    \{x\in\{0,1\}^\#: (vx)^L_{L+t}\ne a\text{~or~} b \text{~for~}
    0<t<|x|, (vx)^L_{L+|x|} = a\}. 
  \end{align}
  If $a\ne b$, then the following are true.
  \begin{subenum}
  \item\label{i:nonempty1}
    $\Pi^a_{b,a} \ne \emptyset$.
  \item\label{i:nonempty1x}
    For any $v\in \{0,1\}^L\setminus\{a,b\}$, $\Pi^a_{v,b} \cup
    \Pi^b_{v,a} \ne \emptyset$ and $\Pi^v_{a,b}\cup \Pi^v_{b,a} \ne
    \emptyset$.
  \item\label{i:nonempty2}
    $\Pi^a_{a,b}\ne\emptyset\Iff \{a,b\}\ne\{(1-x)x^{L-1},
    x^{L-1}(1-x)\}$, where $x=0$ or 1.
  \end{subenum}
\end{prop}
\begin{proof}
  \ref{i:nonempty1}
  Let $E_m = \{(X)^L_L=a, (X)^L_{L+m-1}=b\}$.  Since $(X)^L_{L+t-1}$,
  $t\ge1$, is an irreducible time-homogeneous Markov chain on
  $\{0,1\}^L$, $n=\min\{m: \prob(E_m)>0\}$ exists.  Every realization
  of $(X)^n_n$ in $E_n$ is an element of $\Pi^a_{b,a}$.

  \ref{i:nonempty1x}
  Following the above proof but with $E_m = \{(X)^L_L=a$ or $b,
  (X)^L_{L+m-1}=v\}$, $\Pi^a_{v,b}\cup \Pi^b_{v,a}\ne\emptyset$.
  Next, $\Pi^{a^R}_{v^R,b^R}\cup \Pi^{b^R}_{v^R,a^R}\ne\emptyset$.
  Since $\Pi^v_{a,b} = \{\gamma^R: \gamma\in \Pi^{a^R}_{v^R,b^R}\}$
  and likewise for $\Pi^v_{b,a}$, then $\Pi^v_{a,b}\cup
  \Pi^v_{b,a}\ne\emptyset$.
  
  \ref{i:nonempty2}
  We follow an argument in \cite{chi:26:spl}.  Suppose
  $\Pi^a_{a,b}=\emptyset$.  Then for any $w\in \{\rx\}\cup 
  \{0,1\}^\#$, $awa$ contains $b$.  Fix $n>L$.  Without loss of
  generality, suppose the last letter of $b$ is 1.  Then, as $b$
  appears in $a0^na$ and is not equal to $a$, it must start within
  $0^n$ and end within the last $a$.  In particular, its first letter
  has to be 0.  Then, as $b$ appears in $a 0^n 1^n a$, it
  must appear within $0^n 1^n$, so $b = 0^s 1^{L-s}$ for some $1\le
  s<L$.  Then, as $b$ appears in $a(01)^n a$, $s$ has to be 1 or
  $L-1$.  If $s=1$, then as $b=01^{L-1}$ appears in $a0^n a$ and $aa$,
  $a$ has to be $1^{L-1}0$.  If $s=L-1$, then as $b = 0^{L-1}1$
  appears in $a1^na$ and $aa$, $a$ has to be $10^{L-1}$.  It is easy
  to check that in either case, $\Pi^a_{a,b}=\emptyset$.
\end{proof}

Recall that for a prefix-free set $S\subset\{0,1\}^\#$,
\[
  \prob(S):=\prob\{(X)^n_n\in S\text{~for some~} n\ge1\} =
  \sum_{x\in S} 2^{-|x|}
\]
and if $S\ne\emptyset$, then for each word $x\in S$ of length $l$,
\[
  \prob(x\gv S)
  := \prob\{(X)^l_l = x\gv (X)^n_n\in S\text{~for some~} n\} =
  \frac{2^{-l}}{\prob(S)}.
\]

From its definition, $\Pi^v_{a,b}$ is prefix-free.  It is easy to see
that, if $|v|=l$, then letting $T = \min\{t>l: (X)^L_t = a \text{~or~}
b\}$,
\[
  \prob(\Pi^v_{a,b}) = \prob\{(X)^L_T = a\gv (X)^l_l=v\}.
\]

For $a$ and $b$ in \Cref{t:darboux-1,t:darboux-0}, since they satisfy
\eqref {e:assumption}, from \Cref{p:nonempty}, $\Pi^a_{a,b}\ne
\emptyset$ and $\Pi^a_{b,a}\ne\emptyset$.  As a result, $p\in (0,1)$.

To prove \Cref{t:darboux-0}, the following Darboux-type
of Tauberian theorem will be used (cf.\ \cite{lalley:01:cm,
  bender:74:siam}).  Denote $\udisc = \{z\in\Coms: |z|<1\}$.  For
Borel set $S\subset \Coms$, denote by $\bar S$ its closure and for
$r>0$, $r S = \{rz: z\in S\}$.
\begin{theorem} \label{t:darboux}
  Let $G(z) = \sumzi n a_n z^n$ be a power series with radius of
  convergence $r$.  Suppose that $G$ has an analytic extension around
  every $z\ne r$ with $|z|=r$, and that $G(z) = A(z)(1 - z/r)^\alpha +
  B(z)$ in $(r\udisc)\cap U$, where $U$ is an open set containing $r$,
  $A(z)$ and $B(z)$ are analytic in $U$, $A(r)\ne 0$, and
  $\alpha\in\Reals\setminus\{0, -1, -2, \ldots\}$.  Then
  \[
    a_n = \frac{A(r)}{\Gamma(-\alpha) n^{1+\alpha} r^n} +
    o\Grp{\nth{n^{1+\alpha} r^n}} \quad
    \text{as~} n\toi.
  \]
\end{theorem}

To deal with $C_n(a) = N_{(\eno X n)}(a)$, we will construct stop
times based on the following.
\begin{definition}
  Let $x=(\eno x n)\in \{0,1\}^n$.  Let $I$ be a set of consecutive
  integers between 1 and $n$.  If $N_{(\eno x t)}(a)>N_{(\eno x t)}(b)$
  for $t\in I$, then $I$ is called a period in $x$ where $a$ dominates
  $b$, or simply an $a$-period.  On the other hand, if
  $N_{(\eno x t)}(a) = N_{(\eno x t)}(b)$ for $t\in I$, then $I$ is
  called a neutral period (in $x$ for $a$ and $b$).  It is clear that
  \begin{align} \label{e:dominant}
    N_x(a)>N_x(b)\Iff \text{$n\in$ an $a$-period},\
    N_x(a)=N_x(b)\Iff \text{$n\in$ a neutral period}.
  \end{align}
\end{definition}

\section{Generating function}\label{s:gf-step1}
In view of \Cref{t:darboux}, to prove the asymptotics of
\[
  \psi^v_n(a,b) = \prob\{C_n(a)>C_n(b)\gv (X)^L_L=v\}-1/2
\]
asserted in \Cref{t:darboux-0}, in this and next sections, we will
compute 
\begin{align} \label{e:diff-pgf}
  \Psi^v_{a,b}(z):=\sumoi n \psi^v_n(a,b) z^n.
\end{align}
The calculation in this section ends up with \eqref{e:main-gf}.  In
next section, the explicit form of the main component in the
expression in \eqref{e:main-gf} is obtained.  The radius of
convergence of the power series \eqref{e:diff-pgf} at least 1, as all
$\psi^v_n(a,b)$ are in $[-1/2,1/2]$.  In \cref{s:singular}, the
singular point(s) of $\Psi^v_{a,b}$ on the boundary of the disc of
convergence will be analyzed to yield the asymptotics in
\Cref{t:darboux-0}.

\begin{definition}\label{def:G}
  For $c,e=\pm1$, denote by $G_{c,e}$ the law of $|\gamma|$ with
  $\gamma \sim \prob(\, \cdot\gv \Pi^{\kappa(c)}_{\kappa(e),
    \kappa(-e)})$, where
  \[
    \kappa(i)= \begin{cases}
      a & \text{if~} i = 1,
      \\
      b & \text{if~} i = -1.
    \end{cases}
  \]
  Denote $g_{c,e}(z)=\mean_{\zeta\sim G_{c,e}}(z^\zeta)$, the
  generating function of $G_{c,e}$.  On the other hand, for $v\in 
  \{0,1\}^L\setminus\{a,b\}$, denote $\pi_{v,e} =
  \prob(\Pi^v_{\kappa(e), \kappa(-e)})$, and if $\pi_{v,e}>0$, denote by
  $G_{v,e}$ the law of $|\gamma|$ with $\gamma\sim \prob(\,\cdot \gv
  \Pi^v_{\kappa(e), \kappa(-e)})$ and  $g_{v,e}$ its generating
  function, while if $\pi_{v,e}=0$, define $G_{v,e}\equiv0$ and
  $g_{v,e}\equiv0$.  
\end{definition}

Note that if $v\not\in\{a,b\}$, then $\pi^v_c = p^v_c$ in
\eqref{e:Lambda-deriv-diff}.  On the other hand, if $v =
\kappa(c)\in\{a,b\}$, then $p^v_c=1=1-p^v_{-c}$ while in general
$\pi^v_c$ is less than 1.

From the strong Markov property of $((X)^L_{L+t-1}, t\ge1)$, for
$k>1$, conditional on $(Y_{k-1}, Y_k)$, $T_k - T_{k-1} \sim
G_{Y_{k-1}, Y_k}$ and is independent of $\{Y_i, T_j - T_{j-1}, i\ne
k-1, k, j\ne k\}$.  Likewise, if $v\in\{0,1\}^L \setminus\{a,b\}$,
then conditional on $(X)^L_L=v$ and $Y_1$, $T_1$ is independent of
$\{Y_i, T_i - T_{i-1}, i>1\}$ and $T_1 -L\sim G_{v,Y_1}$.  However, if
$v=\kappa(e)\in \{a,b\}$, then $\prob\{Y_1=e,T_1=L\gv
(X)^L_L=v\}=1$.

Since $((X)^L_{L+t-1}, t\ge1)$ is an irreducible Markov chain, there
are $C>0$ and $\varrho>1$ such that for $n\ge1$ and $v\in \{0,1\}^L$,
\begin{align} \label{e:spec-r}
  \prob\{T\ge n \gv (X)^L_L = v\} \le C\varrho^{-n}, \text{~where~}
  T = \min\{t>L: (X)^L_t = a \text{~or~} b\}.
\end{align}
\begin{lemma} \label{l:rc-G}
  Let $\varrho>1$ be as in \eqref{e:spec-r}  For $x\in\{0,1\}^L$ and
  $c,e=\pm1$, $g_{x,e}$ and $g_{c,e}$ are analytic in $\varrho\udisc$.
\end{lemma}
\begin{proof}
  Let $\zeta\sim G_{c,e}$.  For $k\ge1$, $\prob\{\zeta=k\}
  \prob(\Pi^{\kappa(c)}_{\kappa(e),\kappa(-e)})=
  \prob\{T_2 - T_1=k, (X)^L_{T_2}=\kappa(e)\gv (X)^L_{T_1} =
  \kappa(c)\}$.  Then from \eqref{e:spec-r}, $\prob\{\zeta=k\} = 
  O(\varrho^{-k})$, yielding the proof for $g_{c,e}$.  The proof for
  $g_{x,e}$ is similar.
\end{proof}

Let $S_0=0$ and $S_k = S_{k-1} + Y_k$ for $k\ge0$.  It is seen that
\[
  S_k = \cum Y k = C_{T_k}(a) - C_{T_k}(b), \quad k\ge0.
\]
Define
\[
  \tau_0 = 0, \quad \tau_j = \min\{k>\tau_{j-1}: S_k=0\}, \quad j\ge1.
\]
It can be shown that
\[
  \prob\{\Lsup_{k\toi} S_k = -\Linf_{k\toi} S_k= \infty\}=1.
\]
Since $S_k$ can only change by $\pm1$ each step, this gives
$\prob\{\tau_j<\infty \text{~for all~} j\ge 1\}=1$.  Let
\begin{gather}
  \sigma_0=1,\ \theta_j = T_{\tau_{j-1}+1},\
  \omega_j = S_{\tau_{j-1}+1},\
  \sigma_j = T_{\tau_j},  \quad j\ge1.
  \label{e:sigma-theta-omega}
\end{gather}

\begin{prop} \label{p:stop-times}
  For $j\ge1$,
  \begin{align} \label{e:theta2}
    C_{\theta_j}(a) - C_{\theta_j}(b) = \omega_j
    &= Y_{\tau_{j-1}+1}
    =
    \begin{cases}
      1 & \text{if~} (X)^L_{\theta_j}=a \\
      -1 & \text{if~} (X)^L_{\theta_j} = b
    \end{cases}
    \\\label{e:theta2b}
    &=-Y_{\tau_j}
    =
    \begin{cases}
      -1 & \text{if~} (X)^L_{\sigma_j}=a \\
      1 & \text{if~} (X)^L_{\sigma_j} = b.
    \end{cases}
  \end{align}
  Furthermore,
  \begin{gather} \label{e:theta-sigma}
    \theta_j =
    \min\{t>\sigma_{j-1}: C_t(a) \ne C_t(b)\},
    \quad
    \sigma_j = \min\{t>\theta_j: C_t(a) = C_t(b)\},
  \end{gather}
  i.e., $\{\sigma_{j-1}, \ldots, \theta_j-1\}$ is a neutral period and
  $\{\theta_j, \ldots, \sigma_j-1\}$ is a $\kappa(\omega_j)$-period. 
\end{prop}

\begin{proof}
  From \eqref{e:Y-X}, for $t\ge1$, letting $k(t) = \max\{k: T_k\le
  t\}$, 
  \begin{align} \label{e:Ct}
    C_t(a) - C_t(b) = \sumoi k Y_k\cf{T_k\le t} = S_{k(t)}.
  \end{align}
  By $|a|=|b|=L>1$, $k(1)=0$.  From \eqref{e:Ct}, for $j\ge1$,
  $C_{\theta_j}(a) -  C_{\theta_j}(b) = S_{\tau_{j-1}+1} =
  S_{\tau_{j-1}} + Y_{\tau_{j-1}+1} = Y_{\tau_{j-1}+1}$, giving
  \eqref{e:theta2}; meanwhile, for $j\ge0$, $C_{\sigma_j}(a)
  - C_{\sigma_j}(b) = S_{\tau_j}=0$.  These equations combined with
  \eqref{e:sigma-theta-omega} imply that $\sigma_{j-1} < \theta_j<
  \sigma_j$ for all $j\ge1$.  Given $j$, for $\sigma_{j-1}\le t <
  \theta_j$, by $k(t) = \tau_{j-1}$, $C_t(a) - C_t(b) =0$, so the
  first equation in \eqref{e:theta-sigma} holds.  Next, for
  $\theta_j\le t < \sigma_j$, by $\tau_{j-1}+1\le k(t) < \tau_j$,
  $C_t(a) - C_t(b) = S_{k(t)}\ne0$, so the second equation in
  \eqref{e:theta-sigma} holds.  Finally, since $S_k$ can only change
  by $\pm1$ each step, by definition of $\tau_j$, $S_{\tau_j-1}$ has
  the same sign as  $\omega_j$, while by $S_{\tau_j}=0$, $S_{\tau_j-1}
  = -Y_{\tau_j}$, so $S_{\tau_j-1}$ can only be 1 or $-1$.  It follows
  that $S_{\tau_j-1} = \omega_j$.  Then $Y_{\tau_j} = -\omega_j$,
  giving \eqref{e:theta2b}.
\end{proof}

From \Cref{p:stop-times}, the neutral periods and $a/b$-dominant
periods are interlacing with durations $\theta_j - \sigma_{j-1}$ and
$\sigma_j - \theta_j$, respectively.  The distributions of the
durations are described as follows.
\begin{prop} \label{p:duration}
  $\inum\omega$ is a Markov chain of $\pm1$'s such that $\omega_1=Y_1$
  and
  \begin{align} \label{e:omega}
    \prob\{\omega_{j+1}=-1\gv \omega_j=-1\} =
    \prob\{\omega_{j+1}=1\gv \omega_j=1\} = q, \quad j\ge1.
  \end{align}
  Fix $n\ge1$.  Let $E=\{(X)^L_L=v, \omega_1=s_1, \ldots,
  \omega_n=s_n\}$, where $v\in\{0,1\}^L$ and $s_j=\pm1$.  Then
  $\prob(E)>0 \Iff\prob^v\{\omega_1=s_1\} >0$, and if this holds,
  then conditional on $E$, $\theta_j - \sigma_{j-1}$, $\sigma_j
  - \theta_j$, $j=1,\ldots, n$, are independent such that
  \begin{subenum}
    \item \label{i:dur-gap0}
      the distribution of $\theta_1$ only depends on $v$ and
      $s_1$, 
  \item \label{i:dur-gap}
    $\theta_j - \sigma_{j-1}\sim  G_{-s_{j-1}, s_j}$ for $j>1$, and
  \item \label{i:dur-excursion}
    $\sigma_j -\theta_j\sim \Phi_{s_j}$ for $j\ge1$, where for
    $c=\pm1$, $\Phi_c$ is the law of $\eta = \cum[2] \eta {\tau_1}$
    conditional on $Y_1 = c$ with $\eta_i\sim G_{Y_{i-1}, Y_i}$ being
    independent conditional on $\eno Y {\tau_1}$.
  \end{subenum}
\end{prop}

\begin{proof}
  By the strong Markov property of $(X)^L_{L+t-1}$, $\inum\omega$ is
  a Markov chain.  From  \Cref{p:stop-times}, for $c=\pm1$,
  $\prob\{\omega_{j+1}=c\gv \omega_j=c\}= \prob\{Y_{\tau_j+1} = c\gv
  Y_{\tau_j} = -c\}$, which is equal to $q$ by the strong Markov
  property of $Y_k$.  Since $\prob(E) = 2^{-L}\prob^v\{\omega_1=s_1\}
  \prod^{n-1}_{j=1} \prob\{\omega_{j+1} =s_{j+1} \gv \omega_j =
  s_j\}$, then from \eqref{e:omega}, $\prob(E)>0\Iff \prob^v\{\omega_1
  = s_1\}>0$.

  The conditional independence of $\theta_j - \sigma_{j-1}$, $\sigma_j
  - \theta_j$, $j=1,\ldots,n$, as well as part \ref{i:dur-gap0}
  follows from the strong Markov property of $(X)^L_{L+t-1}$.  For
  $j>1$, $\theta_j - \sigma_{j-1} = \min\{t>0: (X)^L_{\sigma_{j-1} +t}
  = a \text{~or~} b\}$.  From \Cref{p:stop-times},
  $(X)^L_{\sigma_{j-1}} = \kappa(-\omega_{j-1})$.  Then by the strong
  Markov property of $(X)^L_{L+t-1}$, the distribution of $\theta_j -
  \sigma_{j-1}$ conditional on $(X)^L_L=v$ and $\omega_i = s_i$, $1\le
  i\le n$, is the same as the one conditional on $(X)^L_{\sigma_{j-1}}
  = \kappa(-s_{j-1})$ and $(X)^L_{\theta_j} = \kappa(s_j)$.  By the
  time-homogeneity of $(X)^L_{L+t-1}$, the latter conditional
  distribution is $G_{-s_{j-1},s_j}$.  Then part \ref{i:dur-gap}
  follows.

  Part \ref{i:dur-excursion} is also a consequence of the strong
  Markov property and time-homogeneity of $(X)^L_{L+t-1}$.  For
  $j\ge1$, by \Cref{p:stop-times}, $(X)^L_{\theta_j} =
  \kappa(\omega_j)$ and $(X)^L_{\sigma_j} = \kappa(-\omega_j)$.  Then
  the distribution of $\sigma_j - \theta_j$ conditional on $(X)^L_L=v$
  and $\omega_i=s_i$, $1\le i\le n$, is the same as that
  of $\sigma_1 - \theta_1$ conditional on $(X)^L_{\theta_1} =
  \kappa(s_j)$.  Since $\theta_1 = T_1$, letting $\eta_i = T_i -
  T_{i-1}$, $\sigma_1 - \theta_1 = \cum[2]\eta{\tau_1}$.  Let
  $c_1=s_j$ and $\eno[2] c{\tau_1}$ be a set of possible values of
  $\eno[2] Y{\tau_1}$, i.e.,$\prob\{Y_i = c_i, 1<i\le \tau_1\gv
  Y_1=c_1\}>0$.  For each $i$, $Y_i = c_i$ is equivalent to
  $(X)^L_{T_i} = \kappa(c_i)$.  Then the distribution of $\eta_i$
  conditional on $Y_i = c_i$, $1\le i\le \tau_1$ is the same as that
  of $\eta_i$ conditional on $Y_{i-1}=c_{i-1}$ and $Y_i = c_i$, which
  is $G_{c_{i-1}, c_i}$.  Then part \ref{i:dur-excursion} follows.
\end{proof}

With the above preparations, the first step of calculation of
$\Psi^v_{a,b}(z)$ can be carried out.  For ease of notation, in the
following, a function of $z$, say $f(z)$ will be written as $f$ after
it is defined.  For $c=\pm1$, denote by $\phi_c(z) = \mean(z^\eta\gv
Y_1=c)$ the generating function of $\Phi_c$ defined in
\Cref{p:duration} \ref{i:dur-excursion}.  Then $\phi_c$ is analytic in 
$\udisc$ and
\begin{align}\label{e:bridge-pgf}
  \phi_c = \mean[\mean(z^\eta\gv \eno Y{\tau_1})\gv Y_1=c]
  = \mean[g_{Y_1, Y_2}\cdots g_{Y_{\tau_1 - 1} Y_{\tau_1}}\gv
  Y_1=c].
\end{align}
From \eqref{e:dominant} and \Cref{p:stop-times},
\[
  \cf{C_n(a) > C_n(b)} = 
  \sumoi j \cf{\theta_j \le n < \sigma_j,\, \omega_j=1}, \quad n\ge1.
\]
Then for $z\in\udisc$,
\[
  \sumoi n \cf{C_n(a)>C_n(b)} z^n
  =
  \sumoi j \cf{\omega_j=1}\sum^{\sigma_j-1}_{n=\theta_j} z^n
  =
  \sumoi j \cf{\omega_j=1}\frac{z^{\theta_j} - z^{\sigma_j}}{1-z}.
\]
Denote by $\prob^v$ and $\mean^v$ the probability and expectation
conditional on $(X)^L_L=v$, respectively.  Applying $\mean^v$ to both
sides of the above display yields
\[
  \sumoi n \prob^v\{C_n(a)>C_n(b)\} z^n
  =
  \nth{1-z}\sumoi j \mean^v(z^{\theta_j} - z^{\sigma_j}, \omega_j=1).
\]
From $z^{\theta_j} - z^{\sigma_j} = z^{\theta_j}(1-z^{\sigma_j -
  \theta_j})$ and \Cref{p:duration}
\[
  \sumoi n \prob^v\{C_n(a)>C_n(b)\} z^n
  =
  \frac{1-\phi_1}{1-z}
  \sumoi j \mean^v(z^{\theta_j},\omega_j=1).
\]
Then from \eqref{e:diff-pgf}
\begin{align} \label{e:diff}
  \Psi^v_{a,b} =
  \nth{1-z}\Sbr{(1-\phi_1)
    \sumoi j \mean^v(z^{\theta_j},\omega_j=1) - \nth2
  }.
\end{align}  
For $c=\pm1$, let
\begin{align} \label{e:g*}
  g^v_c(z) = \mean^v(z^{\theta_1}, \omega_1=c).
\end{align}
For $j>1$ and $\eno s j\in\{\pm1\}$ with
$\prob^v\{\omega_1 = s_1\}>0$, from \Cref{p:duration},
\begin{align*}
  &
  \mean^v(z^{\theta_j}, \omega_1=s_1, \ldots, \omega_j=s_j)
  \\
  &=
  \mean^v(z^{\theta_{j-1} + (\sigma_{j-1}
  - \theta_{j-1}) + (\theta_j - \sigma_{j-1})} \gv \omega_1=s_1, \ldots, \omega_j=s_j)
  \prob^v\{\omega_1=s_1, \ldots, \omega_j=s_j\}
  \\
  &=
  \mean^v(z^{\theta_{j-1}} \gv \omega_1=s_1, \ldots,
  \omega_{j-1}=s_{j-1})
  \mean(z^{\sigma_{j-1}-\theta_{j-1}} \gv \omega_{j-1}=s_{j-1})
  \\
  &\hspace{1cm}
  \times \mean(z^{\theta_j - \sigma_{j-1}}\gv \omega_{j-1} = s_{j-1},
  \omega_j = s_j)
  \prob^v\{\omega_1=s_1, \ldots, \omega_{j-1}=s_{j-1}\}
  \\
  &\hspace{1cm}\qquad
  \times\prob(\omega_j = s_j \gv \omega_{j-1}=s_{j-1})
  \\
  &=
  \mean^v(z^{\theta_{j-1}}, \omega_1=s_1, \ldots, \omega_j=s_{j-1})
  \phi_{s_{j-1}} g_{-s_{j-1}, s_j} 
  \prob\{\omega_j = s_j \gv \omega_{j-1}=s_{j-1}\}.
\end{align*}
Then by induction,
\[
  \mean^v(z^{\theta_j}, \omega_1=s_1, \ldots, \omega_j=s_j)
  =
  g^v{s_1}
  \prod^{j-1}_{k=1}[\phi_{s_k} g_{-s_k, s_{k+1}} \prob\{\omega_{k+1} =
  s_{k+1} \gv \omega_k = s_k\}.
\]
It is easy to see that the equation holds as well when
$\prob^v\{\omega_1 = s_1\} = 0$.  Thus
\begin{align*}
  \mean^v(z^{\theta_j}, \omega_j=1)
  =
  \sum_{\eno s j=\pm1}
  g^v_{s_1} \prod^{j-1}_{k=1}
  [\phi_{s_k} g_{-s_k, s_{k+1}}
  \prob\{\omega_{k+1}=s_{k+1}\gv \omega_k = s_k\}] \cf{s_j=1}.
\end{align*}
Let
\begin{gather*}
  \nu_0(z) =
  \begin{pmatrix}
    g^v_{-1} \\
    g^v_1
  \end{pmatrix},
  \quad
  P(z)
  =
  \begin{pmatrix}
    q \phi_{-1} g_{1, -1} & p\phi_{-1} g_{1,1} \\
    p \phi_1 g_{-1,-1} & q \phi_1 g_{-1,1}
  \end{pmatrix}, \quad
  \nu = \begin{pmatrix}
    0 \\ 1 - \phi_1 \end{pmatrix}
\end{gather*}
Then $(1-\phi_1)\mean^v(z^{\theta_j}, \omega_j=1) = \nu'_0 P^{j-1} \nu$. 
For $z\in \udisc$, the spectral radius of $P(z)$ is strictly less than
1.  Then from \eqref{e:diff},
\[
  \Psi^v_{a,b}
  =
  \nth{1-z}
  \Sbr{
    \nu'_0\sumoi j P^{j-1} \nu - \nth 2
  }
  = \frac{\nu'_0(I - P)^{-1} \nu-1/2}{1-z}.
\]
Then from
\[
  (I - P)^{-1}=\nth{\det(I - P)}
  \begin{pmatrix}
    1-q \phi_1 g_{-1,1}& p\phi_{-1} g_{1,1} \\
    p \phi_1 g_{-1,-1} & 1-q \phi_{-1} g_{1, -1}
  \end{pmatrix},
\]
it follows that
\begin{align} \label{e:main-gf}
  \Psi^v_{a,b} = \nth{1-z}\Sbr{\frac{\Lambda(z)}{\det(I-P)} - \nth2}
  \quad\text{in~} \udisc,
\end{align}
where
\begin{align} \label{e:Lambda}
  \Lambda(z) = [g^v_{-1} p \phi_{-1} g_{1,1} + g^v_1(1-
  q\phi_{-1} g_{1,-1})](1-\phi_1).
\end{align}
By calculation,
\begin{align} \label{e:det}
  \det(I - P)
  =
  1-q(\phi_{-1} g_{1,-1} + \phi_1 g_{-1,1})
  + \phi_{-1} \phi_1 (q^2 g_{1,-1} g_{-1,1} - p^2 g_{1,1} g_{-1,-1}).
\end{align}
To get $\Psi^v_{a,b}$ explicitly, one has to compute $g_{c,e}$,
$c,e=\pm1$, and $\phi_{\pm1}$.  For $g_{c,e}$, this can be done by a
standard method; see \cref{s:first-step}.  For $\phi_{\pm1}$, this is
done in next section.

\section{Explicit form of $\phi_{\pm1}$} \label{s:gf-step2}
This section calculates $\phi_{\pm1}$.  The explicit forms of the
functions are given in \eqref{e:F+} and \eqref{e:F-}, respectively.
To start, note that $|g_{c,e}|<1$ in $\udisc$.  Consider $\phi_1$.  As
$\tau_1$ is an even number,
\[
  \phi_1 = \sumzi n F_n, \quad\text{where~}
  F_n
  = \mean(g_{Y_1, Y_2} \cdots g_{Y_{2n+1}, Y_{2n+2}}, \tau_1 =
      2n+2\gv Y_1 = 1).
\]

First,
\[
  F_0 =\mean(g_{Y_1, Y_2}, \tau_1 = 2\gv Y_1=1)
  = \mean(g_{Y_1,Y_2}, Y_2=-1\gv Y_1 = 1)
  = q g_{1,-1}.
\]
Recall $S_n = \sum^n_{j=1} Y_j$.  For $n\ge1$, define
\[
  \Cal B_n =\{Y_1 = 1, S_j\ge 1 \text{~for~} 1\le j
  \le 2n, S_{2n+1}=1\}.
\]
Since $\{Y_1=1, \tau_1 = 2n+2\}=\Cal B_n\cap \{Y_{2n+2}=-1\}$ and
$\Cal B_n$ implies $Y_{2n+1}=-1$, then
\begin{align*}
  F_n
  &=
  \mean(g_{Y_1, Y_2} \cdots g_{Y_{2n+1}, Y_{2n+2}}, \Cal B_n,
  Y_{2n+2}=-1\gv Y_1 = 1\}
  \\
  &=
  \mean(g_{Y_1, Y_2} \cdots g_{Y_{2n+1}, Y_{2n+2}}, Y_{2n+2}=-1\gv
  \Cal B_n, Y_1 = 1\}\prob\{\Cal B_n\gv Y_1=1\}
  \\
  &=
  \mean(g_{Y_{2n+1}, Y_{2n+2}}, Y_{2n+2}=-1\gv Y_{2n+1}=-1)
  \\
  &\hspace{4cm}
  \times \mean(g_{Y_1, Y_2} \cdots g_{Y_{2n}, Y_{2n+1}}\gv
  \Cal B_n, Y_1 = 1\}\prob\{\Cal B_n\gv Y_1=1\}
  \\
  &=
  p g_{-1,-1} \mean(g_{Y_1, Y_2}\cdots g_{Y_{2n},  Y_{2n+1}}, \Cal
  B_n\gv Y_1=1).
\end{align*}
Let $Y_*\in \{\pm1\}$ such that $Y_*, \inum Y$ form a time-homogeneous
Markov chain.  Define
\begin{align*}
  \Cal H_n
  &= \{S_j\ge0 \text{~for~} 1\le j<2n, S_{2n}=0\}\\
  &= \{Y_1=1, S_j\ge0 \text{~for~} 1<j<2n, S_{2n}=0\}
  \quad\text{and}\\
  h_n
  &=
  \mean(g_{Y_*, Y_1} g_{Y_1, Y_2} \cdots g_{Y_{2n-1}, Y_{2n}},
  \Cal H_n\gv Y_*=1).
\end{align*}
Since $\eta:= \max_{c,e} |g_{c,e}|^2<1$, $|h_n|\le\eta^n$ and so
\begin{align} \label{e:Hdef}
  H := \sumoi n h_n \quad\text{converges absolutely.}
\end{align}

By $\Cal B_n = \{Y_1 = 1, S_{1+j} - Y_1\ge0 \text{~for~} 1\le j <
2n, S_{2n+1} - Y_1=0\}$ and the time-homogeneity of $Y_*, \inum Y$,
$\mean(g_{Y_1, Y_2}\cdots g_{Y_{2n}, Y_{2n+1}}, \Cal B_n\gv Y_1=1) =
h_n$, giving
\begin{align} \label{e:FH}
  \phi_1 = q g_{1,-1} + p g_{-1,-1} H.
\end{align}

Thus it is necessary to compute $H$.  To this end, define
\begin{align*}
  \Cal K_n
  &= \{S_j>0 \text{~for~} 1\le j<2n, S_{2n}=0\} \quad\text{and}
  \\
  k^\pm_n
  &=\mean(g_{Y_*, Y_1} g_{Y_1, Y_2} \cdots g_{Y_{2n-1},
    Y_{2n}}, \Cal K_n \gv Y_* = \pm1).
\end{align*}
Since $\inum{\Cal K}$ are disjoint and $|k^\pm_n| < \prob(\Cal K_n\gv 
Y_*=\pm1)$, then
\[
  K^\pm := \sumoi n k^\pm_n \text{~converges absolutely and~}
  |K^\pm|<1. 
\]

For $1\le \seqop m<s=n$, $s\ge1$, define
\[
  \Cal R_{\eno m s} = \{S_j>0 \text{~for~} j\in [1, 2m_s]\setminus
  \{\eno{2 m} s\}, S_{2m_1} = \ldots = S_{2m_s}=0\}
\]
and
\[
  r_{\eno ms}
  =
  \mean(g_{Y_*, Y_1} g_{Y_1, Y_2} \cdots g_{Y_{2m_s-1}, Y_{2m_s}}, 
  \Cal R_{\eno m s}\gv Y_* = 1).
\]
Given $n\ge1$, the events $\Cal R_{\eno ms}$ with $m_s=n$ form a
partition of $\Cal H_n$.  In addition, $\Cal R_n =\Cal K_n$.   Then
$r_n = k^+_n$ and
\[
  h_n= k^+_n + \sum_{1\le \seqop m<s=n, s\ge2} r_{\eno m s}.
\]
Take sum over $n\ge1$,
\begin{align} \label{e:HK0}
  H = K^+ + \sum_{1\le\seqop m<s, s\ge2} r_{\eno m s}.
\end{align}

For each sequence $1\le \seqop m<s$, $s\ge2$, put
\[
  \Cal R_* = \{S_{2m_{s-1} + j} - S_{2m_{s-1}}>0 \text{~for~} 1\le j <
  2 m_s - 2 m_{s-1}, S_{2m_s} - S_{2m_{s-1}}=0\}.
\]
Then $\Cal R_{\eno m s} = \Cal R_{\eno m {s-1}} \cap \Cal R_*$, so
\begin{align*}
  r_{\eno ms}
  &=
  \mean(g_{Y_*, Y_1} g_{Y_1, Y_2} \cdots g_{Y_{2m_s-1}, Y_{2m_s}}, 
  \Cal R_* \gv \Cal R_{\eno m {s-1}}, Y_* = 1)
  \\
  &\quad\times \prob\{\Cal R_{\eno m {s-1}} \gv Y_*=1\}.
\end{align*}
Note that $\Cal R_*\in \Cal F(\eno[2m_{s-1}+1] Y{2m_s})$ and $\Cal
R_{\eno m {s-1}}\in \Cal F(\eno Y {2m_{s-1}})$.  In addition, the
second event implies $Y_{2m_{s-1}}=-1$.  Then the expectation on the
\rhs equals
\begin{multline*}
  \mean(g_{Y_*, Y_1} g_{Y_1, Y_2} \cdots g_{Y_{2m_{s-1}-1},
    Y_{2m_{s-1}}} \gv \Cal R_{\eno m {s-1}}, Y_* = 1)
  \\
  \qquad\times
  \mean(g_{2m_{s-1}, 2m_{s-1}+1}\cdots g_{2m_s-1,2m_s}, \Cal R_*\gv
  Y_{2m_{s-1}}=-1).
\end{multline*}
By time-homogeneity, the second expectation equals
\[
  \mean(g_{Y_* Y_1} g_{Y_1, Y_2} \cdots g_{Y_{2(m_s-m_{s-1})-1},
    Y_{2(m_s - m_{s-1})}}, \Cal K_{2(m_s - m_{s-1})} \gv Y_*=-1)
  =k^-_{m_s - m_{s-1}}.
\]
Then
\begin{align*}
  r_{\eno m s}
  &=
  \mean(g_{Y_*, Y_1} g_{Y_1, Y_2} \cdots g_{Y_{2m_{s-1}-1},
    Y_{2m_{s-1}}} \gv \Cal R_{\eno m {s-1}}, Y_* = 1) k^-_{m_s -
    m_{s-1}}
  \\
  &\quad\times
  \prob\{\Cal R_{\eno m{s-1}}\gv Y_*=1\}
  \\
  &=  \mean(g_{Y_*, Y_1} g_{Y_1, Y_2} \cdots g_{Y_{2m_{s-1}-1},
    Y_{2m_{s-1}}}, \Cal R_{\eno m {s-1}}\gv Y_* = 1) k^-_{m_s - m_{s-1}}
  \\
  & = r_{\eno m{s-1}} k^-_{m_s - m_{s-1}} = \cdots =
  k^+_{m_1} k^-_{m_2-m_1} \cdots k^-_{m_s - m_{s-1}}.
\end{align*}
As a result,
\begin{align}\label{e:HK}
  H=
  K^+ + \sum_{\eno j s\ge 1, s\ge2} k^+_{j_1} k^-_{j_2}\cdots
  k^-_{j_s}
  &= K^+ + K^+ \sum^\infty_{s=2} (K^-)^{s-1} = \frac{K^+}{1
    - K^-}.
\end{align}

To solve $H$, two more equations in $H$ and $K^\pm$ are needed.  To
start,
\begin{align}\nonumber
  k^\pm_1
  &= \mean(g_{Y_* Y_1} g_{Y_1 Y_2}, Y_1=1, Y_2=-1 \gv
  Y_*=\pm1)\\\label{e:k1}
  &= g_{\pm1,1} g_{1,-1} \prob\{Y_1=1, Y_2=-1\gv Y_*=\pm1\}
  = q g_{\pm1,1} g_{1,-1} \prob\{Y_1=1\gv Y_*=\pm1\}.
\end{align}
For $n>1$,  $\Cal K_n = \{Y_1=1\}\cap \Cal H^*_{n-1} \cap\{Y_{2n} =
-1\}$, where $\Cal H^*_k = \{S_{1+j} - Y_1\ge0$ for $1\le j\le 2k,
S_{2k+1} - Y_1=0\}$, $k\ge1$.  Then by Markov property and 
time-homogeneity,
\begin{align*} 
  k^\pm_n
  &=\mean(g_{Y_*, Y_1} g_{Y_1, Y_2} \cdots g_{Y_{2n-1}, Y_{2n}},
  Y_1=1, \Cal H^*_{n-1}, Y_{2n}=-1\gv Y_* = \pm1)
  \\
  &=g_{\pm1,1} \prob\{Y_1=1\gv Y_*=\pm1\}
  \mean(g_{Y_1,Y_2} \cdots g_{Y_{2n-1},Y_{2n}}, \Cal H^*_{n-1},
  Y_{2n}=-1\gv Y_1=1)\\
  &=g_{\pm1,1} \prob\{Y_1=1\gv Y_*=\pm1\}
  \mean(g_{Y_*,Y_1} \cdots g_{Y_{2n-2},Y_{2n-1}}, \Cal H_{n-1},
  Y_{2n-1}=-1\gv Y_*=1).
\end{align*}
Since $\Cal H_{n-1}$ implies $Y_{2n-2}=-1$, by Markov property, the
last expectation equals
\begin{align*}
  &\hspace{-1cm}
  \mean(g_{Y_*,Y_1} \cdots g_{Y_{2n-2},Y_{2n-1}}, 
  Y_{2n-1}=-1\gv \Cal H_{n-1}, Y_*=1) \prob\{\Cal H_{n-1}\gv Y_*=1\}
  \\
  &=\mean(g_{Y_*,Y_1} \cdots g_{Y_{2n-3},Y_{2n-2}}\gv \Cal H_{n-1},
  Y_*=1)
  \\
  &\hspace{2cm}
  \times \mean(g_{Y_{2n-2}, Y_{2n-1}}, Y_{2n-1}=-1\gv
  Y_{2n-2}=-1) \prob\{\Cal H_{n-1}\gv Y_*=1\}
  \\
  &=
  p g_{-1,-1} \mean(g_{Y_*,Y_1} \cdots g_{Y_{2n-3},Y_{2n-2}}, \Cal
  H_{n-1} \gv Y_*=1) = p g_{-1,-1} h_{n-1}.
\end{align*}
Then
\begin{align} \label{e:kn}
  k^\pm_n = p g_{\pm1,1} g_{-1,-1} \prob\{Y_1=1 \gv Y_*=\pm1\}
  h_{n-1}.
\end{align}
Take the sum of \eqref{e:k1} and \eqref{e:kn} over $n\ge1$.  Then two
equations in $H$, $K^+$, and $K^-$ obtain, as desired,
\begin{align} \label{e:KH}
  K^\pm = \prob\{Y_1=1\gv Y_*=\pm1\}g_{\pm1,1}(q g_{1,-1} 
  + p g_{-1,-1} H).
\end{align}
Plug \eqref{e:KH} into \eqref{e:HK} to get
\[
  H = \frac{p g_{1,1}(q g_{1,-1} + p g_{-1,-1} H)}
  {1 - q g_{-1,1} (q g_{1,-1} + p g_{-1,-1} H)},
\]
or 
\begin{align} \label{e:H-q}
  p g_{-1,1} g_{-1,-1} H^2 + (p^2 g_{1,1} g_{-1,-1} + q^2 g_{1,-1}
  g_{-1,1} -1) H/q + p g_{1,-1} g_{1,1}=0.
\end{align}

To solve $H$ from \eqref{e:H-q}, first suppose $g_{-1,1}
g_{-1,-1}\ne0$.  Then \eqref{e:H-q} is a quadratic equation in $H$
with determinant
\begin{align} \label{e:Delta-Q}
  \Delta = Q^2 - 4g_{1,-1} g_{-1,1}, \ \text{where}\ \
  Q=(1-p^2 g_{1,1} g_{-1,-1} + q^2 g_{1,-1} g_{-1,1})/q.
\end{align}
We need to take the square root of $\Delta$.  The next lemma makes
sure that this will not result in singularity in $\udisc$.
\begin{lemma} \label{l:Delta-zero}
  $\Delta(1)=0$.  On the other hand, there is $c>1$, such  that
  $\Delta$ is analytic in $c\udisc$ and $\Delta\ne0$ in
  $c\udisc\setminus\{1\}$.
\end{lemma}

\begin{proof}
  By $g_{\pm1,\pm1}(1)=1$ and $q=1-p$, $\Delta(1) = (1-p^2+q^2)^2 -
  4q^2 = 0$.  Therefore, 1 is a root of $\Delta$. 
  Suppose $z_0\in\bar\udisc\setminus\{1\}$ is also a root of
  $\Delta$.  Then $\theta:= Q(z_0)/2$ is a square root of
  $g_{1,-1}(z_0) g_{-1,1}(z_0)$, which combined with the expression 
  of $Q$ in \eqref{e:Delta-Q} yields $2\theta= (1 - p^2 g_{1,1}(z_0)
  g_{-1,-1}(z_0) + q^2 \theta^2)/q$,
  i.e.,
  \[
    (1-q\theta)^2 = p^2 g_{1,1}(z_0) g_{-1,  -1}(z_0).
  \]
  Since $|g_{\pm1,\pm1}|<1$ in $\udisc$, it follows that $|z_0|=1$ and
  $g_{1,1}(z_0) g_{-1,-1}(z_0) = g_{-1,1}(z_0) g_{1,-1}(z_0)=1$.
  Since $z_0\ne1$, then there is an integer $d\ge2$ such that
  $z^d_0=1$, and there is an integer $d_{c,e}$ for each pair of $c$
  and $e$ such that the support of $G_{c,e}$ is a subset of $d_{c,e} +
  d\Ints$.  By \eqref{e:assumption}, we can assume without loss of
  generality that $a\ne1^L$ and $b\ne1^L$, for otherwise we can make
  the transform $a\to\bar a$, $b\to\bar b$, which leaves $g_{c,e}$
  unchanged.  By \Cref{p:nonempty}~\ref {i:nonempty1x}, there is
  $c\in\{\pm1\}$ such that $\Pi^{\kappa(c)}_{1^L,\kappa(-c)}\ne
  \emptyset$.  Let $\gamma_1$ be a path in the set.  Likewise, there
  is $e\in\{\pm1\}$ such that $\Pi^{1^L}_{\kappa(e),\kappa(-e)}\ne
  \emptyset$.  Let $\gamma_2$ be a path in the set.  Then
  $\gamma_1\gamma_2 \in\Pi^{\kappa(c)}_{\kappa(e),\kappa(-e)}$.
  However, this implies that for any $n\ge0$, $\gamma_1 1^n\gamma_2
  \in \Pi^{\kappa(c)}_{\kappa(e),\kappa(-e)}$.  Since $e=c$ or $e=-c$,
  then either the support of $G_{c,c}$ or that of $G_{c,-c}$ contains
  all large integers, which is a contradiction.

  As a result, $\Delta\ne0$ in $\bar\udisc\setminus\{1\}$.  From
  \Cref{l:rc-G}, $\Delta$ is analytic in $\varrho\udisc$ for some 
  $\varrho>1$.  For each $z$ with $|z|=1$, whether or not it is 1,
  there is $r_z>0$ such that $B(z,r_z):=\{w\in \Coms: |w-z|<r_z\}
  \subset \varrho\udisc$ and $\Delta\ne0$ in $B(z,r_z)\setminus\{1\}$.
  There are $\eno z n$ with each $|z_i|=1$, such that the boundary of
  $\udisc$ can be contained in the union of $B(z_i, r_{z_i})$.  Then
  there is $c>1$ such that $c\Delta$ is contained in the union of
  $\udisc$ and $B(z_i, r_{z_i})$.  It is easy to see $c\Delta$ has the
  stated property.
\end{proof}

Let $\Omega = c\udisc\setminus [1,\infty)$, where $c\udisc$ is as in 
\Cref{l:Delta-zero}.  Then $\sqrt\Delta$ can be defined analytically in
$\Omega$ such that at $z=0$, its values is $Q(0) = 1/q$.  Then the two
solutions to
\eqref{e:H-q},
\[
  \nth{2 p g_{-1,1} g_{-1,-1}} \times
  \{(1-p^2 g_{1,1} g_{-1,-1} - q^2 g_{1,-1} g_{-1,1})/q \pm \sqrt\Delta\}
\]
are analytic in $\Omega\cap\{z: g_{-1,1}(z) g_{-1,-1}(z)\ne0\}$.   The
region is open and contains 0.  To see which solution is $H$, let
$z\to0$.  Note that $g_{c,e}(z)\to0$ for $c,e=\pm1$.  Then from
\eqref{e:Hdef}, $H(z)\to0$.  On the other hand, between the two 
solutions, only the one with $-\sqrt\Delta$ in its expression tends to
0.  Therefore, that solution is $H$.  Then from \eqref{e:FH},
\[
  \phi_1 
  = q g_{1,-1} + 
  \frac{(1-p^2 g_{1,1} g_{-1,-1} - q^2 g_{1,-1} g_{-1,1})/q - \sqrt\Delta}
  {2 g_{-1,1}} = \frac{Q-\sqrt\Delta}{2g_{-1,1}}.
\]
Denote
\begin{align} \label{e:V-f-h}
  V = \frac{Q - \sqrt\Delta}2, \quad
  f_{c,e} = \nth{g_{c,e}}.
\end{align}
Then
\begin{align}\label{e:F+}
  \phi_1= f_{-1,1} V.
\end{align}

It remains to consider the case where $g_{-1,1} g_{-1,-1}=0$.  This
can be divided into two cases.  First, $g_{-1,1}=0$.  Then
\eqref{e:H-q} gives
\[
  H = \frac{pq g_{1,-1} g_{1,1}}{1-p^2 g_{1,1} g_{-1,-1}}
\]
and so from \eqref{e:FH},
\[
  \phi_1 = q g_{1,-1} + p g_{-1,-1} \times \frac{pq g_{1,-1}
    g_{1,1}}{1-p^2 g_{1,1} g_{-1,-1}} = \frac{q g_{1,-1}}{1 - p^2
    g_{1,1} g_{-1,-1}}.
\]
It is seen this is the limit of \eqref{e:F+} as $g_{-1,1}\to0$.
Second, $g_{-1,1}\ne0=g_{-1,-1}$.  Then directly from \eqref{e:FH},
$\phi_1 = q g_{1,-1}$, which is the value of \eqref{e:F+} by letting
$g_{-1,-1}=0$.  As a result, \eqref{e:F+} still holds.  It follows
that $\phi_1$ is analytic in $\Omega$.

From \eqref{e:V-f-h}, $V$ is invariant when $-1$ and $1$ are switched.
Then by symmetry,
\begin{align} \label{e:F-}
  \phi_{-1} = f_{1,-1} V
\end{align}
which is analytic in $\Omega$.

\section{Analysis of singularity} \label{s:singular}
As noted in \cref{s:gf-step1}, $\Psi^v_{a,b}$ is analytic in $\udisc$.
We need the following.
\begin{prop} \label{p:Phi-analytic}
  $\Psi^v_{a,b}$ can be analytically extended to an open set containing
  $\bar\udisc\setminus\{1\}$.
\end{prop}
\begin{proof}
  As stated at the end of \cref{s:gf-step2}, $\phi_{\pm1}$ can be
  analytically defined in an open set $\Omega$ containing
  $\bar\udisc\setminus\{1\}$.  Then from \eqref{e:Lambda} and
  \eqref{e:det}, $\Lambda$ and $\det(I-P)$ can be analytically
  extended $\Omega$.  Therefore, by \eqref{e:main-gf}, it suffices to
  show that $\det(I - P(z))\ne0$ for every $z\ne1$ with $|z|=1$.
  Assume there is $z_0\ne1$ with $|z_0|=1$ such that $\det(I -
  P(z_0))=0$.  Then there is a vector $x\ne0$ such that $P(z_0) x =
  x$.  Let
  \[
    P^* = \begin{pmatrix} \pi_{11} & \pi_{12} \\ \pi_{21} & \pi_{22}
    \end{pmatrix} =
    \begin{pmatrix}
      q |\phi_{-1}(z_0) g_{1,-1}(z_0)| &
      p |\phi_{-1}(z_0) g_{1,1}(z_0)| \\
      p |\phi_1(z_0) g_{-1,-1}(z_0)| &
      q |\phi_1(z_0) g_{-1,1}(z_0)|
    \end{pmatrix}
  \]
  and $z$ the vector of absolute values of the coordinates of $x$.
  Then every coordinate of $P^*z$ is greater or equal to the
  corresponding coordinate of $z$.  By $|\phi_{\pm1}(z_0)|\le1$
  and $|g_{\pm1,\pm1}(z_0)|\le1$, Perron--Frobenius theorem implies
  that this holds only if $|g_{\pm1,\pm1}(z_0)|=1$.  However, from the
  proof of \Cref{l:Delta-zero}, this is impossible.
\end{proof}

From \Cref{p:Phi-analytic}, to prove \Cref{t:darboux-0}, it only
remains to check the singularity of $\Psi_{a,b}$ at 1.  Combining
\eqref{e:Lambda}, \eqref{e:det}, \eqref{e:F+}, and \eqref{e:F-},
\begin{align*}
  \Lambda
  &= [g^v_{-1} p f_{1,-1} g_{1,1} V + g^v_1(1-qV)] (1-f_{-1,1}V)
  \\
  &=g^v_{-1} p f_{1,-1} g_{1,1}(V - f_{-1,1}V^2)
  + g^v_1[1-(q+f_{-1,1})V + q f_{-1,1}V^2]
\end{align*}
and
\[
  \det(I - P)
  =
  1-2qV + (q^2 - p^2 g_{1,1} g_{-1,-1} f_{1,-1} f_{-1,1})V^2.
\]
By \eqref{e:Delta-Q}, $V^2=QV - g_{1,-1} g_{-1,1}$.  Then
\[
  \Lambda
  = g^v_{-1} p f_{1,-1} g_{1,1}[g_{1,-1} + (1-f_{-1,1}Q)V]
  + g^v_1[1-q g_{1,-1} -(q+f_{-1,1}-qf _{-1,1}Q)V]
\]
and
\begin{align*}
  \det(I - P)
  &= 1 - 2qV + (q^2 - p^2 g_{1,1} g_{-1,-1} f_{1,-1} f_{-1,1})(Q V -
  g_{1,-1} g_{-1,1})
  \\
  &=
  2 - qQ + [-2q+(q^2 - p^2 g_{1,1} g_{-1,-1} f_{1,-1} f_{-1,1})Q] V,
\end{align*}
where the last equation is due to the definition of $Q$ in
\eqref{e:Delta-Q}.  Plug in \eqref{e:V-f-h} to get
\begin{align} \label{e:Lambda-det-Delta}
  \Lambda = \Lambda_1 + \lfrac{\Lambda_2\sqrt\Delta}2,
  \quad
  \det(I - P) = D_1 + D_2\sqrt\Delta/2,
\end{align}
where
\begin{align*}
  \Lambda_1&=
  g^v_{-1} p f_{1,-1} g_{1,1}[g_{1,-1} + (1-f_{-1,1}Q)Q/2]
  \\
  &\quad
  + g^v_1[1-q g_{1,-1} -(q+f_{-1,1}-qf _{-1,1}Q)Q/2],  
  \\
  \Lambda_2
  &= -g^v_{-1} p f_{1,-1} g_{1,1}(1-f_{-1,1}Q)+ g^v_1(q+f_{-1,1}-qf
  _{-1,1}Q)
\end{align*}
and
\begin{align*}
  D_1
  &=2-[4q-(q^2 - p^2 g_{1,1} g_{-1,-1} f_{1,-1} f_{-1,1})Q] Q/2,
  \\
  D_2
  &= 2q-(q^2 - p^2 g_{1,1} g_{-1,-1} f_{1,-1} f_{-1,1})Q.
\end{align*}
Then from \eqref{e:main-gf}
\[
  \Psi^v_{a,b} =
  \nth{1-z}\Grp{\frac{\Lambda_1 + \Lambda_2\sqrt\Delta/2}
  {D_1 + D_2\sqrt\Delta/2} -\nth2}\quad\text{in~} \udisc.
\]
Multiply the numerator and denominator on the \rhs by $D_1 -
D_2\sqrt\Delta/2$ to get
\begin{align} \label{e:main-gf3}
  \Psi^v_{a,b}
  =
  A + \frac{B}{C\sqrt\Delta} \quad\text{in~} \udisc,
\end{align}
where
\[
  A = 
  \frac{(\Lambda_1 - D_1/2)D_1  - \Delta(\Lambda_2 - D_2/2)D_2/4}
  {(1-z)(D^2_1 - D^2_2\Delta/4)}, \quad
  B = \frac{\Lambda_2 D_1 - \Lambda_1 D_2}{2(1-z)}, \quad
  C = \frac{D^2_1}\Delta - \frac{D^2_2}4.
\]

From \Cref{l:rc-G}, $\Lambda_i$, $D_i$, $i=1,2$, $Q$, and $\Delta$ are 
analytic in $\varrho\udisc$, where $\varrho>1$.  From 
\[
  g^v_c(1) = p^v_c, \quad
  g_{c,e}(1) = 1, \quad c,e=\pm1,
\]
it follows that $\Lambda_1(1) = D_1(1)=0$ and $\Lambda_2(1)= D_2(1)/2=
p$.  As a result, $A$, $B$, and $C$ can be analytically extended
in a neighborhood around 1 with
\[
  B(1)
  =
  -(\Lambda_2 D_1 - \Lambda_1 D_2)'(1)/2
  = p [\Lambda'_1(1)-D'_1(1)/2]
\]
and
\[
  C(1) = \lim_{z\to1} [D^2_1(z)/\Delta(z) - D^2_2(z)/4] = -p^2\ne0.
\]

To evaluate $B(1)$, notice that $g'_{c,e}(1) = m_{c,e}$, where
$m_{c,e}$ is defined in \eqref{e:mean-gap}.  Since $p m_{c,c} =
\mu_{c,c}$ and $q m_{c,-e} = \mu_{c,-c}$, then from \eqref{e:Delta-Q}
\[
  Q(1) = 2, \quad
  Q'(1) = - p (\mu_{1,1} + \mu_{-1,-1})/q + (\mu_{1,-1} + \mu_{-1,1}).
\]
Since the two bracketed terms in the expression of $\Lambda_1$ are
equal to 0 at $z=1$, then
\begin{align*}
  \Lambda'_1(1)
  &= p^v_{-1} p(g_{1,-1} + Q/2 -f_{-1,1}Q^2/2)'(1)
  \\
  &\quad
  +p^v_1(1-q g_{1,-1} -qQ/2-f_{-1,1}Q/2+qf _{-1,1}Q^2/2)'(1)
  \\
  &= p^v_{-1} p[m_{1,-1} + 2m_{-1,1} - 3Q'(1)/2]
  \\
  &\quad
  + p^v_1[-qm_{1,-1} + (1-2q) m_{-1,1} +(3q/2- 1/2)Q'(1)]
  \\
  &=
  \Grp{p^v_{-1} - \frac{3q}2+\nth 2}
  \frac pq \mu_* +
  \frac{p(\mu_{-1,1} - \mu_{1,-1})}{2q},
\end{align*}
where $\mu_*=\mu_{1,1} + \mu_{-1,-1} + \mu_{1,-1} + \mu_{-1,1}$.  On
the other hand,
\begin{align*}
  D'_1(1)
  &=
  (-2qQ + q^2 Q^2/2 - p^2 g_{1,1} g_{-1,-1} f_{1,-1} f_{-1,1}Q^2/2)'(1)
  \\
  &=
  -2p^2(m_{1,1} + m_{-1,-1} - m_{1,-1} - m_{-1,1}) -2pQ'(1)
  =
  -2(q-p)\frac p q \mu_*.
\end{align*}
Then
\begin{align*}
  B(1)
  &=
  -\frac{p^2}q\Sbr{(\lfrac p2 - p^v_{-1})\mu_* +\frac{\mu_{1,-1} -
      \mu_{-1,1}}2}.
\end{align*}
From \eqref{e:sum-mean}, $\mu_*=2^L$.  Then $B(1)=-p^2\beta^v_{a,b}/q$
according to \eqref{e:beta}.  As a result, $B/C$ can be analytically
extended in a neighborhood around 1 with value at $z=1$ equal to
$\beta^v_{a,b}/q$.  Furthermore, from \eqref{e:Delta-Q}, $\Delta(1) =
0$ and $\Delta'(1) = -4(p/q)\mu_*= 2^{L+2}p/q$.  Then $\Delta/(1-z)$
has an  analytic extension around $z=1$ with the value at $z=1$ equal to
$-\Lambda'(1)>0$.  If $B(1)\ne0$, then by Darboux \Cref{t:darboux},
the asymptotics in \eqref{e:psi-n} holds.  If $B(1)=0$ but
$B(z)\not\equiv0$, then $B\Sp k(1)\ne0$ for some $k\ge1$.  Then again
by Darbous \Cref{t:darboux}, \eqref{e:psi-n} holds.  If $B(z)\equiv0$,
then $\Psi^v_{a,b}$ is analytic in an open set containing
$\bar\udisc$, so by Darbous \Cref{t:darboux}, \eqref{e:psi-n} holds.


\appendix
\section{First-step analysis} \label{s:first-step}
The appendix describes how to obtain $p$ in \eqref{e:tran-p},
$\mu_{c,e}$ in \eqref{e:mean-gap}, $p_c$ in \eqref
{e:Lambda-deriv-diff}, and $g_{c,e}$ in \Cref{def:G}, using a 
method  sometimes called the
first-step analysis (cf.\ \cite{lalley:01:cm, chi:26:spl}).  Define
\[
  T = \min\{n\ge1: (X)^L_{L+n} = a \text{~or~} b\}
\]
and for $x\in \{0,1\}^L$ and $y=a,b$,
\[
  \gamma_{x,y}(z) = \sumoi n \prob\{T=n, (X)^L_{L+n} = y\gv
  (X)^L_L=x\} z^n.
\]
From $\gamma_{x,y}$ all the above quantities and functions can be
derived:
\begin{gather} \label{e:gamma-z}
  p = \gamma_{a,a}(1) = \gamma_{b,b}(1), \quad
  \mu_{c,e} = \gamma'_{\kappa(c),\kappa(e)}(1), \quad c,e=\pm1,
  \\\nonumber
  p_c = 2^{-L} + \sum_{x\in \{0,1\}^L\setminus\{a,b\}}
  \gamma_{x,\kappa(c)}(1), \quad c=\pm1,
  \\\nonumber
  g_{c,c}(z) = \gamma_{\kappa(c), \kappa(c)}(z)/p, \quad
  g_{c,-c}(z) = \gamma_{\kappa(c), \kappa(-c)}(z)/(1-p), \quad c=\pm1,
\end{gather}
where $\kappa(1)=a$ and $\kappa(-1)=b$.  To get, for example
$\gamma_{x,a}(z)$, denote by $\pi_{x,y}$ the transition probabilities
of the Markov chain $(X)^L_{L+n-1}$.  Then $\gamma_{x,a}(z) =
z[\pi_{x,a} +  \sum_{y\ne a,b} \pi_{x,y} \gamma_{y,a}(z)]$.  Let
$\bar\gamma_a(z)$ be the (column) vector of $\gamma_{x,a}(z)$,
$\bar\pi_a$ the vector of $\pi_{x,a}$, and $Q_{a,b}$ the matrix of
$\pi_{x,y}\cf{y\ne a,b}$.  Then the above equation can be written as
$\bar\gamma_a(z) = z\bar\pi_a + z Q_{a,b} \bar\gamma_a(z)$, giving
\[
  \bar\gamma_a(z) = z(I - z Q_{a,b})^{-1}\bar\pi_a.
\]
By $\det(I-Q_{a,b})\ne0$ (cf.\ \cite{chi:26:spl}),
$\bar\gamma_a(1) = (I-Q_{a,b})^{-1}\bar\pi_a$ and $\bar\gamma'_a(1) =
(I-Q_{a,b})^{-1}\bar\gamma_a(1)$.

To get $Q_{a,b}$, index each $x=(\eno x L)\in \{0,1\}^L$ by 
$n_x=1+\sum^L_{i=1} 2^{i-1}x_i$, so that the matrix $(\pi_{x,y})$ can
be written as 
\[
  \small
  \begin{pmatrix}
    1/2&1/2& & &  & & \\
    & &1/2&1/2&  & & \\
    & & & &\ddots& & \\
    & & & &&1/2&1/2\\
    1/2&1/2& & &  & & \\
    & &1/2&1/2&  & & \\
    & & & &\ddots& & \\
    & & & &&1/2&1/2\\
  \end{pmatrix}.
\]
Then $Q_{a,b}$ can be obtained from the matrix by replacing all the
entries in the $n_a$- and $n_b$-th columns with 0.

\begin{proof}[Proof of \eqref{e:sum-mean}]
  Let $a,b\in\{0,1\}^L$ and $a\ne b$.  Let $T_a = \min\{n\ge1:
  (X)^L_{L+n}=a\}$ and
  \[
    f_a(z) = \sumoi n \prob\{T_a=n\gv (X)^L_L = a\} z^n
    = \sum_{x\in\Pi_a} 2^{-|x|} z^{|x|},
  \]
  where $\Pi_a = \{x\in\{0,1\}^\#: (ax)^L_{L+t}\ne a \text{~for~}
  0<t<|x|, (ax)^L_{L+|x|}=a\}$, i.e., the set of paths leading to the
  first return to $a$ given the initial word being $a$.  Depending
  on whether such a path visits $b$, it is either a path in
  $\Pi^a_{a,b}$, or a path of the form $z_1 \cset y k z_2$ with
  $k\ge0$, where $z_1\in \Pi^a_{b,a}$, $y_i\in\Pi^b_{b,a}$, and
  $z_2\in\Pi^b_{a,b}$.  As a result,
  \[
    f_a = \gamma_{a,a} + \sumzi k \gamma_{a,b} (\gamma_{b,b})^k
    \gamma_{b,a}
    = \gamma_{a,a} + \frac{\gamma_{a,b}\gamma_{b,a}}{1-\gamma_{b,b}}.
  \]
  Take derivatives on both sides at $z=1$ and apply \eqref{e:gamma-z}, 
  noting that from \Cref{p:nonempty}, $1-\gamma_{b,b}(1) =
  \gamma_{b,a}(1)>0$.  Then $f'_a(1) = \sum_{c,e=\pm1} \mu_{c,e}$.  On
  the other hand, $f'_a(1)$ is the mean waiting time until $a$ is
  revisited by $(X)^L_{L+t-1}$ given that $(X)^L_L=a$, and it is the
  reciprocal of the long-term frequency of $a$, which is $2^{-L}$.
  Then $f'_a(1)=2^L$, giving \eqref{e:sum-mean}. 
\end{proof}

\end{document}